\documentclass[11pt,reqno]{amsart}

\usepackage{amsthm}
\usepackage{amsmath}
\usepackage{amssymb}
\usepackage{pstricks, pst-node}

\usepackage{graphicx}

\theoremstyle{plain}
 \newtheorem{theorem}{Theorem}[section]
 \newtheorem{lemma}[theorem]{Lemma}
  \newtheorem{proposition}[theorem]{Proposition}
 \newtheorem{corollary}[theorem]{Corollary}

 \newtheorem{failed-conjecture}[theorem]{Failed Conjecture}
\theoremstyle{definition}
  \newtheorem{definition}[theorem]{Definition}
  \newtheorem{example}[theorem]{Example}
  
  \newtheorem{counter-example}[theorem]{Counter Example}  
\theoremstyle{remark}
  \newtheorem{remark}[theorem]{Remark}

\newcommand{\dflag}{d^{\, \sigma}}
\newcommand{\dpflag}{d^{\, \pi}}

\begin{document}

\date{}
\title{Relations on Generalized Degree Sequences}
\author{Caroline J.\ Klivans}
\address{Departments of Mathematics and Computer Science,
The University of Chicago,
Chicago, IL 60637}
\author{Kathryn L.\ Nyman}
\address{Department of Mathematics and Statistics,
Loyola University Chicago,
Chicago, IL 60626}
\author{Bridget E.\ Tenner}
\address{Department of Mathematical Sciences,
DePaul University,
Chicago, IL 60614}


\begin{abstract}
 We study degree sequences for simplicial posets and 
 polyhedral complexes, generalizing
 the well-studied graphical degree sequences.  Here we extend the more
 common generalization of vertex-to-facet degree sequences by
 considering arbitrary face-to-flag degree sequences.  In particular,
 these may be viewed as natural refinements of the flag $f$-vector of
 the poset.  We investigate properties and relations of these
 generalized degree sequences, proving linear relations between flag
 degree sequences in terms of the composition of rank jumps of the
 flag.  As a corollary, we recover an $f$-vector inequality on
 simplicial posets first shown by Stanley.
 \end{abstract}

\maketitle

\section{Introduction}

Degree sequences of graphs, which record how many edges contain each
vertex, have been studied extensively, and enjoy a rich literature.  One
popular chapter of this literature is the characterization of when an integer
sequence can be a degree sequence; for example, see \cite{Hoogeveen}.

In higher dimensions, the notion of degree sequence for simplicial
complexes has received considerably less attention.  For example, although
the behavior of the vertex-to-facet degree sequence for pure simplicial
complexes has been studied in \cite{Duval, KR, Murthy}, little is known
about the intrinsic properties or characterizations for this, possibly
most natural, extension of graphical degree sequences.

In this paper, we formulate a higher dimensional and more general
analogue of graphical degree sequences, and we study the nature of these
sequences for simplicial posets and for more general polyhedral complexes.
In particular, we define the face-to-flag degree sequence, recording
how many flags with prescribed rank jumps contain each face of a given
rank.  
More precisely, let $\mathcal{P}$ be a pure, rank-$k$ simplicial poset,
and $\sigma = (\sigma_1, \ldots, \sigma_m)$ a composition of $k$.  The
\emph{face-to-flag degree sequence} $d^{\, \sigma}(\mathcal{P})$ is a
sequence indexed by the faces $F_i$ of rank $\sigma_1$, with the
corresponding entry in $d^{\, \sigma}(\mathcal{P})$ counting the number of
flags
\begin{equation*}
 \{F_i \subset X_2 \subset \cdots \subset X_m : {\rm rk}(X_j) =
 \sigma_1 + \sigma_2 + \cdots + \sigma_j \},
\end{equation*}
\noindent which are the flags containing $F_i$ with rank jumps $\sigma_1,
\ldots, \sigma_m$.
We give linear relations between such sequences, proving a majorization
result in the case of simplicial posets, in terms of the relative sizes of the rank jumps of the flags.
Namely, for
$\pi = (\pi_1, \ldots, \pi_m)$ a permutation of $\sigma$, if
$\pi_1 \geq \sigma_1$, then
$$d^{\, \sigma} \unrhd d^{\, \pi}.$$
Furthermore, we obtain a weak majorization result for face-to-flag degree sequences analogously defined for complexes whose maximal faces are simple polytopes.

Our majorization result in the simplicial case  
yields results on $f$-vectors and flag $f$-vectors of pure simplicial posets because these sequences naturally refine the
flag $f$-vector of the poset.   In particular, we recover a result of Stanley: for a rank-$k$
simplicial poset, the $f$-vector satisfies the inequality $f_i \leq
f_{k-i}$, when $i \le k-i$.

In Section~\ref{section:deg seq}, we review graphical degree sequences
as a basis and motivation for studying generalized degree sequences,
also recalling the definition of vertex-to-facet degree sequences and
a few preliminary results.  In Section~\ref{section:gen deg seq}, we
introduce the face-to-flag degree sequences, and prove our main result
(Theorem \ref{thm:facetoflag}), giving a total ordering (via
majorization) of the face-to-flag degree sequences of a simplicial
poset.  As a corollary, we recover the aformentioned result of Stanley.
 In Section~\ref{section:non-simp deg seq}
we consider non-simplicial posets.  First, Theorem
\ref{thm:facetoflag} is extended to the setting of complexes all of
whose maximal faces are simple polytopes.  Finally, we give an
elementary example showing that these results may not be extended to
arbitrary polyhedral complexes.

\section{Degree sequences, motivation and preliminaries}\label{section:deg seq}

\subsection{Graphical degree sequences}

\begin{definition}
For a simple graph $G = (V,E)$ with $|V| = n$,  the \emph{graphical degree
sequence} of $G$ is the sequence
$$d(G) = (d_1, d_2, \ldots, d_n),$$
where
$$d_i = |\left\{j : \{i,j\} \in E \right\}|,$$
and the vertices are labeled so that $d_1 \ge d_2 \ge \ldots \ge d_n$.
\end{definition}

Majorization (or dominance) order is often the natural choice for
comparing graphical degree sequences, as in~\cite[Chapter 7]{Marshall}.

\begin{definition}
Given two sequences of nonnegative integers $a = (a_1, a_2, \ldots, a_n)$ and
$b = (b_1, b_2, \ldots, b_m)$ having the same sum,  $a$ \emph{majorizes} $b$,  written $a \unrhd b$, if the
following system of inequalities hold:
$$ a_1 \geq b_1 $$
$$ a_1 + a_2 \geq b_1 + b_2 $$
$$ \vdots $$
$$ a_1 + a_2 + \cdots + a_{n-1} \geq b_1 + b_2 + \cdots + b_{n-1}.$$
\noindent If one weakens the sum requirement to the inequality $\sum_i a_i
\geq \sum_i b_i $, then $a$ is said to \emph{weakly majorize} $b$, written $a \succeq b$.  
Majorization forms a partial ordering on integer sequences. 
\end{definition}

The \emph{conjugate} of a sequence $d = (d_1, \ldots, d_n)$ of nonnegative integers is the sequence
$d^T$ given by $$d_i^T = |\{j : d_j \geq i \}|$$ for $i \ge 1$.  This agrees with the
transpose when considering $d$ as an integer partition.  An elementary relation on graphical degree sequences compares a sequence and its conjugate.

\begin{proposition}[{\cite[Chapter 7, Section D]{Marshall}}]
\label{prop:conjugate}
The degree sequence of any simple graph $G$ is majorized by its conjugate:
$$\big(d(G)\big)^T \unrhd d(G).$$
\end{proposition}

The problem of characterizing graphical degree sequences, that is, deciding which
sequences arise as the degree sequence of some graph, is well understood.  For
example,~\cite{Hoogeveen} lists seven criteria, all of which are linear relations on
the degree sequence or between the degree sequence and its conjugate.

\subsection{Vertex-to-facet degree sequences}\ 
Moving to higher dimensions, we first consider \emph{pure, rank-$k$}
simplicial complexes; that is, complexes in which every maximal face is
of rank $k$, where we define the rank of a face to be one more than its dimension.

For a pure, rank-$k$ simplicial complex, the ``vertex-to-facet'' degree
sequence, counting the number of facets containing each vertex, has
received the most attention as a generalization of graphical degree
sequences.  For example, these sequences are treated in~\cite{Duval}, \cite{KR}, and \cite{Murthy},
where they arise in connection with Laplacian eigenvalues, plethysms,
and zonotopes respectively.

Little is known about the intrinsic properties of such integer
sequences.  In general, the vertex-to-facet sequence does not share
many of the nice properties of the graphical case.  For example,
Proposition~\ref{prop:conjugate} no longer holds.

\begin{example}
Let $\Delta$ be the complete, rank-$3$ complex on five vertices.  Then the
vertex-to-facet degree sequence is $(6,6,6,6,6)$, and $(6,6,6,6,6)^T =
(5,5,5,5,5,5)$.  Proposition~\ref{prop:conjugate} does not hold here,
since the majorization occurs in the wrong direction: $(6,6,6,6,6) \unrhd
(5,5,5,5,5,5)$.
\end{example}

To match the notation for generalized degree sequences introduced in
the upcoming Definition~\ref{defn:face to flag}, let $d^{\, (\sigma_1,
  \sigma_2)}(\Delta)$ be the sequence counting, for each rank-$\sigma_1$ face in a
simplicial complex $\Delta$, the number of rank-$(\sigma_1+\sigma_2)$ faces
containing that face.  Duval, Reiner, and Dong consider
face-to-facet sequences, and give the following relation for a rank-$k$ complex.

\begin{proposition}[{\cite[Proposition 8.4]{Duval}}] \label{drd}
For any pure, rank-$k$ simplicial complex $\Delta$,
\begin{equation*}
\left(d^{\, (i,k-i)}(\Delta)\right)^T \unrhd d^{\, (k-i,i)} (\Delta).
\end{equation*}
\end{proposition}

In fact, a more direct relationship exists between $d^{\, (i,k-i)}$ and
$d^{\, (k-i,i)}$, and between more general ``face-to-flag'' 
degree sequences where the compositions of rank jumps for the flags are
permutations of each other.  These generalized degree sequences, indexed
by compositions of rank jumps, are defined in the subsequent section.

\section{Generalized degree sequences}\label{section:gen deg seq}

\subsection{Face-to-flag degree sequences}\ 

For the remainder of the paper, we will assume that all simplicial
complexes and posets are pure.  That is, we assume that all maximal faces have the same rank.

\begin{definition}
A \emph{simplicial poset} is a poset with a minimal element $\hat{0}$ in which every principal order ideal $[\hat{0},x]$ is a boolean algebra.
\end{definition}
In parallel to the terminology for the more commonly
studied vertex-to-facet degree sequences of simplicial complexes, we
will also use the term \emph{face} to denote an element of the poset,
and \emph{flag} to denote a nested sequence of faces (that is, a chain) in the poset.
Additionally, when no confusion can arise, we will move freely between
these two interpretations.

\begin{definition}\label{defn:face to flag}
For a pure, rank-$k$ simplicial poset $\mathcal{P}$, and $\sigma = (\sigma_1, \ldots, \sigma_m)$ a
composition of $k$,  the \emph{face-to-flag degree sequence}
$\dflag(\mathcal{P})$ is
\begin{equation*}
\dflag(\mathcal{P}) = \left(\dflag(F_1), \dflag(F_2), \ldots, \dflag(F_s)\right),
\end{equation*}
\noindent where $\{F_1, F_2, \ldots, F_s\}$ are the rank-$\sigma_1$ faces of $\mathcal{P}$, and
\begin{equation*}
\dflag(F_i) = |\{F_i \subset X_2 \subset \cdots \subset X_m  : {\rm rk}(X_j) =
\sigma_1 + \sigma_2 + \cdots + \sigma_j \}|.
\end{equation*}
\noindent   The sequence $\dflag(\mathcal{P})$ records the
degrees of rank-$\sigma_1$ faces to flags with rank jumps $\sigma_1, 
\ldots, \sigma_m$.  As with graphical degree sequences, for a fixed $\sigma$, the faces $F_i$ are indexed so that
$\dflag(F_1) \ge \dflag(F_2) \ge \ldots \ge \dflag(F_s)$.
 We assume that all flags end at the top rank, since otherwise
one could take the appropriate truncation of the poset.
\end{definition}

Note that there is an abuse of notation in Definition~\ref{defn:face to flag}: the function $\dflag$ is defined both on a simplicial poset and on a face, which itself could be considered as a simplicial poset.  Throughout this article, the context of the usage should be clear.

\begin{example}
\label{ex:motivating}
Let $\Delta$ be the pure, rank-$4$ simplicial complex with facets $\{1234,
 1246, 1256\}$.  Consider the following three degree sequences recording
vertex-to-edge/tetrahedron flags,
vertex-to-triangle/tetrahedron flags, and edge-to-triangle/tetrahedron
flags, respectively.
\begin{eqnarray*}
d^{\, (1,1,2)}(\Delta) &=& (9,9,6,6,3,3)\\
d^{\, (1,2,1)}(\Delta) &=& (9,9,6,6,3,3)\\
d^{\, (2,1,1)}(\Delta) &=& (6,4,4,4,4,2,2,2,2,2,2,2)
\end{eqnarray*}
From here it is easy to see that
\begin{equation*}
d^{\, (1,1,2)}(\Delta) = d^{\, (1,2,1)}(\Delta) \unrhd d^{\, (2,1,1)}(\Delta).
\end{equation*}
\end{example}

To formalize the majorization behavior observed in Example~\ref{ex:motivating}, comparing face-to-flag degree sequences for a given complex, fix $\mathcal{P}$ a rank-$k$
simplicial poset, and $\sigma = (\sigma_1, \ldots, \sigma_m)$ a composition of $k$.  Additionally, let $\sum \dflag (\mathcal{P})$ denote the sum of the entries in $\dflag(\mathcal{P})$.

Let $\{F_1, \ldots, F_s\}$ be the rank-$\sigma_1$ elements of $\mathcal{P}$.  Observe that
\begin{eqnarray}
 \dflag(\mathcal{P}) = \binom{k-\sigma_1}{\sigma_2, \ldots, \sigma_m} \left( d^{\, (\sigma_1,k-\sigma_1)}(F_1), \ldots, d^{\, (\sigma_1,k-\sigma_1)}(F_s) \right).
 \label{eqn:binomial}
 \end{eqnarray}

\noindent Thus, the sum of the entries in $\dflag(\mathcal{P})$ is 
\begin{eqnarray}\label{eqn:sum of deg seq}
\sum \dflag(\mathcal{P}) &=& \binom{k-\sigma_1}{\sigma_2, \ldots, \sigma_m} \sum_{l=1}^s d^{\, (\sigma_1,k-\sigma_1)}(F_l) 
	=  \binom{k-\sigma_1}{\sigma_2, \ldots, \sigma_m} f_{\sigma_1,k}(\mathcal{P})\nonumber \\ 
	& = &\binom{k-\sigma_1}{\sigma_2, \ldots, \sigma_m} f_k(\mathcal{P}) \binom{k}{\sigma_1} = f_k(\mathcal{P}) \binom{k}{\sigma_1, \ldots, \sigma_m},
\end{eqnarray}
 where $f_k(\mathcal{P})$ is the number of rank-$k$ elements of $\mathcal{P}$, and $$f_{\sigma_1,k}(\mathcal{P}) =|\{X_1 \subset X_2: {\rm rk}(X_1)=\sigma_1, {\rm rk}(X_2) = k\}|$$ is an entry in the flag-$f$ vector of $\mathcal{P}$.
 Additionally, since the right side of equation~\eqref{eqn:sum of deg seq} only depends on the entries of $\sigma$, we see that $\sum \dflag(\mathcal{P}) = \sum \dpflag(\mathcal{P})$ for any permutation $\pi$ of the composition $\sigma$.

\begin{lemma}
Fix $\sigma = (\sigma_1, \ldots, \sigma_m)$ a composition of  $k$ and $\pi$ a permutation of $\sigma$, with $\pi_1 > \sigma_1$.  If $F$ is a rank-$\sigma_1$ element of a pure, rank-$k$ simplicial poset $\mathcal{P}$, and $G > F$ is a rank-$\pi_1$ element, then $\dflag(F) \geq \dpflag(G)$.  
\label{lemma:F<G}
\end{lemma}

\begin{proof}

Let $\pi_1=\sigma_l$ for some $1< l \le m$.
\begin{eqnarray*}
	\dflag(F) & = & d^{\, (\sigma_1,k-\sigma_1)}(F) \binom{k-\sigma_1}{\sigma_2, \ldots, \sigma_m} 
	 =  \mathop{\sum_{H: {\rm rk}(H)=k}}_{ F < H} \binom{k-\sigma_1}{\pi_1} \binom{k-\sigma_1-\pi_1}{\sigma_2, \ldots, \widehat{\sigma_l},		\ldots, \sigma_m}  \\
	&\geq&   \mathop{\sum_{H: {\rm rk}(H)=k}}_{ G < H} \binom{k-\sigma_1}{\pi_1} \binom{k-\sigma_1-\pi_1}{\sigma_2, \ldots, 		\widehat{\sigma_l},\ldots, 	\sigma_m} \\
	&= &   \frac{\binom{k-\sigma_1}{k-\sigma_1-\pi_1}}{\binom{k-\pi_1}{k-\pi_1-\sigma_1}} \mathop{\sum_{H: {\rm rk}(H)=k}}_{ G < H} \binom	{k-\pi_1}{\sigma_1} \binom{k-\sigma_1-\pi_1}{\sigma_2, \ldots, \widehat{\sigma_l},\ldots, \sigma_m}  \\
	&\geq &  \mathop{\sum_{H: {\rm rk}(H)=k}}_{ G < H} \binom{k-\pi_1}{\pi_2, 
	\ldots, \pi_m}  = d^{\, (\pi_1,k-\pi_1)}(G) \binom{k-\pi_1}{\pi_2, \ldots, \pi_m} =  \dpflag(G).
\end{eqnarray*}

\noindent The last inequality is due to the fact that $\pi_1 >\sigma_1$, and $\{\sigma_1,
\ldots, \widehat{\sigma_l}, \ldots, \sigma_m\}$ and $\{\pi_2, \ldots, \pi_m\}$ are
equal as sets.
\end{proof}

\begin{theorem}
\label{thm:facetoflag}

Fix a pure, rank-$k$ simplicial poset $\mathcal{P}$, a composition $\sigma = (\sigma_1, \ldots, \sigma_m)$ of $k$, and $\pi = ( \pi_1, \ldots, \pi_m)$ a permutation of $\sigma$.  If $\pi_1  \geq \sigma_1$, then 
\[
\dflag(\mathcal{P}) \unrhd \dpflag(\mathcal{P}).
\]

\end{theorem}

\begin{proof}  

We recall that $\sum \dflag(\mathcal{P})$ depends only on the entries of $\sigma$, and therefore $\sum \dflag(\mathcal{P}) = \sum \dpflag(\mathcal{P})$.
We divide the rest of the proof into two cases: $\pi_1 = \sigma_1$ and $\pi_1> \sigma_1$.

First, consider the case $\pi_1 = \sigma_1$. Here we are considering elements of
the same rank compared to different flags.  By equation~\eqref{eqn:binomial}, we see that $\dflag(\mathcal{P})=\dpflag(\mathcal{P})$.

Now, suppose $\pi_1>\sigma_1$.  Let $\{F_1, F_2, \ldots, F_s\} $ be the rank-$\sigma_1$ elements of $\mathcal{P}$, and let $\{G_1, G_2, \ldots, G_t\}$ be the rank-$\pi_1$ elements.  Furthermore, assume they are labeled so that 
\[ \dflag(F_i) \geq \dflag(F_j) \text{ and } \dpflag(G_i) \geq \dpflag(G_j) \]
for all $i<j$. 

By Lemma \ref{lemma:F<G}, if $F < G$, then $\dflag(F) \geq \dpflag(G)$.
It remains to prove that
\begin{equation}\label{eqn:inequality in proof}
\sum_{l=1}^{r} \dpflag(G_l) \leq \sum_{l=1}^{r} \dflag(F_l) \, \, \, \textrm{for all} \,\,\, r \leq s.
\end{equation}

We prove this by induction on $r$.  Suppose $r = 1$, and consider the smallest
$j$ such that $F_j < G_1$. Then we have
$$ \dflag(F_1) \geq \dflag(F_j) \geq \dpflag(G_1).$$

For $r > 1$, we consider two subcases.  

Suppose there exists $F_j < G_l$ with $j > r$ and $l
\leq r$.  As above, we have
$$ \dflag(F_r) \geq \dflag(F_j) \geq \dpflag(G_l) \geq \dpflag(G_r).$$
In this case, the $r$th terms in the sums satisfy the necessary inequalities themselves,
hence the $r$th partial sum, equation~\eqref{eqn:inequality in proof}, holds by induction. 

Now assume there does not exist $F_j < G_l$ with $j>r$ and $l \leq r$.  We prove this case directly.  First, note that if we sum the degrees of all rank-$\pi_1$ elements which are greater than a given rank-$\sigma_1$ element, we count the degree of the fixed rank-$\sigma_1$ element $\binom{\pi_1}{\sigma_1}$ times:  
$$ \sum_{G_l : F_j < G_l} \dpflag(G_l) = \binom{\pi_1}{\sigma_1} \dflag(F_j). $$

\noindent This implies that
$$
\begin{aligned}
\binom{\pi_1}{\sigma_1} \sum_{j = 1}^{r} \dflag(F_j) &= \sum_{j=1}^{r}
\sum_{G_l : F_j < G_l} \dpflag( G_l) \\ &\geq \sum_{j=1}^{r}
\sum_{\substack{G_l : F_j < G_l \\ l \leq r}} \dpflag(G_l) \\ &= \binom{\pi_1}{ \sigma_1} \sum_{j=1}^{r} \dpflag(G_j).
\end{aligned}
$$
\end{proof}

As a corollary, we obtain an inequality for the $f$-vector of a simplicial poset.  This result was first shown by Stanley (see \cite{Hibi} for a brief history of this inequality).

\begin{corollary}
For $\mathcal{P}$ a pure simplicial poset of rank $k$, with $i \le k-i$,
$$ f_{i} \leq f_{k-i}, $$
where $f_i$ is the $i$th entry of the $f$-vector of $\mathcal{P}$, indexed by rank.
\end{corollary}

\begin{proof}

If $a$ and $b$ are nonnegative sequences with $a \unrhd b$, then the number of non-zero entries of
$a$ must be less than or equal to the number of non-zero entries of
$b$.  For the sequence $d^{\, (i,k-i)}$, the number of (non-zero) entries is
equal to the number of rank-$i$ elements of $\mathcal{P}$, because $\mathcal{P}$ is pure of rank $k$.
\end{proof}

\begin{remark}
Note also that:
$$ \sum_{l=1}^s d^{\, (i,j-i)}(F_l) = \sum_{l=1}^s d^{\, (j-i,i)}(F_l) = f_{i,j} = f_{j-i,j}$$
where the $f_{i,j}$ are entries in the flag $f$-vector of $\mathcal{P}$, indexed by ranks. Hence,
we can view generalized degree sequences as refinements of flag
$f$-vectors. The theorem above states that for all pure simplicial posets, and for 
$i<j-i$, the contributions which make up $f_{j-i,j}$ are ``more
spread out''  than those which make up $f_{i,j}$.
\label{remark:SequenceSum}
\end{remark}

The definition of face-to-flag degree sequences can easily be extended
to a notion of flag-to-flag degree sequences which record the degrees
of flags with specified rank jumps to flags with continuing rank
jumps.  The flag-to-flag degree sequence can be computed in terms of
face-to-flag degree sequences.  An analogous majorization result holds in this
case where the necessary condition is in terms of the relative
sizes of the sums of rank jumps of the initial flags.

\section{Non-simplicial degree sequences}\label{section:non-simp deg seq}

The definition of face-to-flag degree sequences 
is not specific to simplicial posets,
 and so we carry over that definition to the
setting of a pure polyhedral complex.  In particular, if all maximal
faces of our complex are \emph{simple}
polytopes, then many of the results from the simplicial case also hold.  For a discussion of simple polytopes, see \cite{Ziegler}.
\begin{definition}
A polytope is \emph{simple} if every proper upper interval in its face lattice, $[x,\hat 1]$ with $x \neq \hat0$, is a boolean algebra.
\end{definition}

As a starting point, we observe that Lemma~\ref{lemma:F<G}
holds in this case.

\begin{lemma}
\label{lemma:cubicalF<G}
Let $\sigma =(\sigma_1, \ldots, \sigma_m)$ be a
  composition of $k$, and $\pi = (\pi_1, \ldots, \pi_m)$ a
  permutation of $\sigma$ with $\sigma_1<\pi_1$.  If $F$ is a
  rank-$\sigma_1$ face of a pure, rank-$k$ complex in which every maximal face is a simple polytope, 
   and if $G \supset F$ is a rank-$\pi_1$ face, then $\dflag(F) \geq \dpflag(G)$.
\end{lemma}
\begin{proof}
Lemma~\ref{lemma:cubicalF<G} follows as in Lemma~\ref{lemma:F<G}.  In
particular, the initial multinomial expression
$$\dflag(F) = d^{\, (\sigma_1,k-\sigma_1)}(F)
\binom{k-\sigma_1}{\sigma_2, \ldots, \sigma_m}$$

\noindent is easy to see, since every upper interval of the face lattice ending at rank $k$ is boolean.
\end{proof}

The simple case differs from the simplicial case in that the 
sums of the degree sequences $\sum \dflag(\mathcal{C})$ and $\sum
\dpflag(\mathcal{C})$ are not necessarily equal for a complex 
$\mathcal{C}$ with simple maximal faces.
However, if every maximal face of $\mathcal{C}$
is the \emph{same} simple polytope, then we establish an inequality on
the sums which implies weak majorization, as demonstrated in the following analogue
to Theorem~\ref{thm:facetoflag}.

\begin{theorem}
\label{thm:simplefacetoflag}
Fix a pure, rank-$k$ complex $\mathcal{C}$ in which every maximal face is the same simple polytope, a
composition $\sigma = (\sigma_1, \ldots, \sigma_m)$ of $k$, and $\pi
= (\pi_1, \ldots, \pi_m)$ a permutation of $\sigma$.  If $\pi_1 \geq
\sigma_1$, then
\[\dflag(\mathcal{C}) \succeq \dpflag(\mathcal{C}).\]

\end{theorem}
  
\begin{proof}
The proof is identical to that of Theorem~\ref{thm:facetoflag}, with
Lemma~\ref{lemma:cubicalF<G} replacing Lemma~\ref{lemma:F<G}, except that equality of the total sums is no longer guaranteed. 

Let $\{F_1, F_2, \ldots, F_s\}$ be the rank-$\sigma_1$ faces of $\mathcal C$.  Then we have
\[ \dflag(\mathcal{C}) =  \binom{k-\sigma_1}{\sigma_2, \ldots, \sigma_m} \left (d^{\, (\sigma_1,k-\sigma_1)}(F_1), \ldots, d^{\, (\sigma_1, k-\sigma_1)}(F_s)\right).\]
By the same argument as in Theorem~\ref{thm:facetoflag},  if $\sigma_1 < \pi_1$, then $$ \sum_{i=1}^{r}
\dpflag(G_i) \leq \sum_{i=1}^{r} \dflag(F_i) \, \, \, \textrm{for all}
\,\,\, r \leq s,$$ 
\noindent where $\{G_1, \ldots, G_t\}$ are the rank-$\pi_1$ faces in $\mathcal{C}$.   

Now,
\begin{eqnarray*}
	\sum \dflag(\mathcal{C}) &=  &\binom{k-\sigma_1}{\sigma_2, \ldots, \sigma_m}   f_{\sigma_1,k}(\mathcal{C})\\
	&= & \binom{k-\sigma_1}{\sigma_2, \ldots, \sigma_m} f_k(\mathcal{C}) N_{\sigma_1},
\end{eqnarray*}
\noindent where $N_{\sigma_1}$ is the number of rank-$\sigma_1$ faces contained in a given rank-$k$ face.

Comparing this to the sum $\sum \dpflag(\mathcal{C})$, we find
\begin{eqnarray}
\frac{\sum \dflag(\mathcal{C})}{\sum \dpflag(\mathcal{C})} &=& \frac{N_{\sigma_1}}{N_{\pi_1}} \cdot \frac{{k- \sigma_1 \choose \pi_1}}{{k - \pi_1 \choose \sigma_1}}.
\label{eqn:weakmajor}
\end{eqnarray}
This ratio must be greater than or equal to $1$ for weak majorization to hold. 

To establish this inequality, we restrict our attention to a single maximal face and consider the number of pairs of rank-$\sigma_1$ faces contained in rank-$\pi_1$ faces.
Since the maximal face is simple, each rank-$\sigma_1$ face is contained in $\binom{k-\sigma_1}{\pi_1-\sigma_1}$ rank-$\pi_1$ faces.  Each rank-$\pi_1$ face contains at least $\binom{\pi_1}{\sigma_1}$ rank-$\sigma_1$ faces, so 
$$ \binom{k-\sigma_1}{\pi_1-\sigma_1} N_{\sigma_1} \geq \binom{\pi_1}{\sigma_1} N_{\pi_1}.$$
This implies that the expression in equation (\ref{eqn:weakmajor}) is at least $1$.
Therefore, if $\sigma_1 < \pi_1$, then 
$$\dflag(\mathcal{C}) \succeq \dpflag(\mathcal{C}).$$
Furthermore if $\sigma_1 = \pi_1$, then 
$$\dflag(\mathcal{C}) = \dpflag(\mathcal{C}).$$
\end{proof}

In fact, the result of Theorem~\ref{thm:simplefacetoflag} would hold for suitably defined ``polyhedral posets'' as well. For example, just as a simplicial poset is one in which every principal order ideal is a boolean algebra, a poset is \emph{cubical} if every principal order ideal is isomorphic to the face lattice of a cube, and such posets also satisfy the results of the theorem.

\subsection{A polyhedral counterexample}\

Theorem \ref{thm:facetoflag} breaks down, even in the case of weak majorization,  as we diverge from the simplicial and simple cases.  For example, in a 3-polytope, the sum of the entries of $d^{(1,3)}$ is equal to $f_1$, the number of vertices of the polytope, while the sum of the entries of $d^{(3,1)}$ is $f_3$, the number of facets.  Thus, any 3-polytope with more facets than vertices results in $d^{(3,1)}$ weakly majorizing $d^{(1,3)}$.  The following example describes one such polyhedral complex; in fact, a Platonic solid.

\begin{example}
\label{ex:failing2}
Let $\Delta$ be the octahedron.  We have
\begin{eqnarray*}
d^{\, (1,3)}(\Delta) &=& (1,1,1,1,1,1), \text{ and}\\
d^{\, (3,1)}(\Delta) &=& (1,1,1,1,1,1,1,1).
\end{eqnarray*}
Although $(3,1) \unrhd (1,3)$, the (weak) majorization of these degree sequences is in the direction opposite to that in the results for the simplicial and simple cases: $d^{\, (3,1)}(\Delta) \succeq d^{\, (1,3)}(\Delta)$.
\end{example}

\section*{Acknowledgments}
The authors thank Lou Billera for helpful discussions and several
important ideas.  We are also
grateful for the advice and perspective of two very thoughtful referees,
particularly for the suggestion to investigate further the case of simple
polytopes.

\end{document}